\documentclass[12pt]{amsart}
\usepackage{amsmath,amssymb,amsfonts,amsthm,amstext}
\usepackage{url,xspace,enumerate,color,graphicx,verbatim}

\newcommand{\Z}{{\mathbb Z}}
\newcommand{\R}{{\mathbb R}}

\renewcommand{\P}{{\mathbb P}}

\newcommand{\E}{{\mathbb E}}

\newcommand\ee{e}

\newcommand\oo{\infty}

\newcommand\De{\Delta}

\newcommand\rad{\text{\rm rad}}

\newcommand\resp{respectively}

\newcommand\Om{\Omega}
\newcommand\qq{\qquad}
\newcommand\pc{p_{\text{\rm c}}}
\newcommand\vpc{\vec p_{\text{\rm c}}}
\newcommand\lip{F}
\newcommand\brwx{BRW($\xi$)}
\newcommand\sbrwx{$C+\text{\brwx}$}

\newcommand\df\textbf

\newcommand{\fr}{\rightarrowtail}
\newcommand{\di}{\Lambda}

\newtheorem{thm}{Theorem}
\newtheorem{lemma}[thm]{Lemma}

\newcounter{mycount}

\newenvironment{romlist}{\begin{list}{\rm(\roman{mycount})}%
   {\usecounter{mycount}\labelwidth=1cm\itemsep 0pt}}{\end{list}}

   \newenvironment{remark}[1][Remark.]{\begin{trivlist}
\item[\hskip \labelsep {\bfseries #1}]}{\end{trivlist}}

\begin{document}
\title{Lipschitz percolation}

\author[Dirr]{N.\ Dirr}
\address[N.\ Dirr]{Department of Mathematical Sciences, University of Bath, Bath BA2 7AY, UK}
\email{n.dirr@maths.bath.ac.uk}
\urladdr{http://www.maths.bath.ac.uk/$\sim$nd235/}

\author[Dondl]{P.\ W.\ Dondl}
\address[P.\ W.\ Dondl]{Hausdorff Center for Mathematics and Institute for
Applied Mathematics, Endenicher Allee 60, D-53115 Bonn, Germany}
\email{pwd@hcm.uni-bonn.de} \urladdr{http://www.dondl.net/}

\author[Grimmett]{G.\ R.\ Grimmett}
\address[G.\ R.\ Grimmett]{Statistical Laboratory, Centre for Mathematical
Sciences, Cambridge University, Wilberforce Road, Cambridge CB3
0WB, UK} \email{g.r.grimmett@statslab.cam.ac.uk}
\urladdr{http://www.statslab.cam.ac.uk/$\sim$grg/}

\author[Holroyd]{\\ A.\ E.\ Holroyd}
\address[A.\ E.\ Holroyd]{\sloppypar Microsoft Research, 1 Microsoft Way, Redmond WA 98052, USA;
and Department of Mathematics, University of British Columbia,
121--1984 Mathematics Road, Vancouver, BC V6T 1Z2, Canada}
\email{holroyd@math.ubc.ca}
\urladdr{http://math.ubc.ca/$\sim$holroyd/}

\author[Scheutzow]{M.\ Scheutzow}
\address[M.\ Scheutzow]{Fakult\"at II, Institut f\"ur Mathematik, Sekr.\ MA 7--5,
Technische Universit\"at Berlin, Strasse des 17.\ Juni 136,
D-10623 Berlin, Germany} \email{ms@math.tu-berlin.de}
\urladdr{http://www.math.tu-berlin.de/$\sim$scheutzow/}

\date{17 November 2009}


\begin{abstract}
We prove the existence of a (random) Lipschitz function
$\lip:\Z^{d-1}\to\Z^+$ such that, for every $x\in\Z^{d-1}$, the
site $(x,\lip(x))$ is open in a site percolation process on
$\Z^{d}$. The Lipschitz constant may be taken to be $1$ when
the parameter $p$ of the percolation model is sufficiently
close to $1$.
\end{abstract}

\maketitle

\section{Introduction}

Let $d \ge 1$ and $p\in (0,1)$. The site percolation model on the
hypercubic lattice $\Z^d$ is obtained by designating each site
$x\in\Z^d$ \df{open} with probability $p$, and otherwise
\df{closed}, with different sites receiving independent states.
The corresponding probability measure on the sample space $\Om =
\{0,1\}^{\Z^d}$ is denoted by $\P_p$, and the expectation by
$\E_p$. We write $\Z^+=\{1,2,\dots\}$, and $\|\cdot\|$ for the
$1$-norm on $\Z^d$.

\begin{thm}\label{main}
For any $d\geq 2$, if $p> 1-(2d)^{-2}$ then there exists
a.s.\ a (random) function $\lip:\Z^{d-1}\to \Z^+$ with the
following properties.
\begin{romlist}
\item For each $x\in \Z^{d-1}$, the site $(x,\lip(x))\in\Z^{d}$
    is open.
\item For any $x,y\in \Z^{d-1}$ with $\|x-y\|=1$ we have
    $|\lip(x)-\lip(y)|\leq 1$.
\item For any isometry $\theta$ of $\Z^{d-1}$ the functions
    $\lip$ and $\lip\circ \theta$ have the same laws, and the
    random field $(\lip(x): x\in\Z^{d-1})$ is ergodic under each translation of
    $\Z^{d-1}$.
\item There exists $A=A(p,d) < \infty$ such that
$$
\P_p(\lip(0)>k)\leq A \nu^k,\qq k\ge 0.
$$
where $\nu=2d(1-p)<1$.

\end{romlist}
\end{thm}

We may think of $((x,\lip(x)): x\in\Z^{d-1})$ as a random
surface, or a Lipschitz embedding of $\Z^{d-1}$ in $\Z^d$. When
$d=2$, the existence of such an embedding for large $p$ is a
consequence of the fact that two-dimensional directed
percolation has a non-trivial critical point. The result is
less straightforward when $d \ge 3$.

The event that there exists an $F$ satisfying (i) and (ii) is
clearly increasing, and invariant under translations of
$\Z^{d-1}$, therefore there exists $p_{\mathrm{L}}$ such that the
event occurs with probability $1$ if $p>p_{\mathrm{L}}$ and $0$ if
$p<p_{\mathrm{L}}$.  Theorem \ref{main} implies that
$p_{\mathrm{L}}\le 1-(2d)^{-2}$. This upper bound may be improved
to $1-(2d-1)^{-2}$ as indicated at the end of Section
\ref{sec:mainest}, but we do not attempt to optimize it here. (A
similar remark applies to the forthcoming Theorem \ref{thm:hill}.)
The inequality $p_{\mathrm{L}}>0$ also holds, because site
percolation on $\Z^d$ with next-nearest neighbour edges has a
non-trivial critical point.

Some history of the current paper, and some implications of the
work, are summarized in Section \ref{bkg}.  In Section
\ref{sec:mainthms} we present a variant of Theorem \ref{main}
involving finite surfaces. The principal combinatorial estimate
appears in Section \ref{sec:mainest}, and the proofs of the
theorems may be found in Section \ref{sec:pf}. Further properties
of Lipschitz embeddings will be presented in \cite{GH09}.

\section{Background and applications}\label{bkg}

The percolation model is one of the most studied models for a
disordered medium, and the reader is referred to \cite{G99} for a
recent account of the theory.  The basic question is to determine
for which values of $p$ there exists an infinite self-avoiding
walk of open sites. There exists a critical value $\pc$, depending
on the choice of underlying lattice, such that such a walk exists
a.s.\ when $p>\pc$, and not when $p<\pc$. It is clear that
$\pc(\Z)=1$, and it is fundamental that $\pc(\Z^d)<1$ when $d \ge
2$. Similarly, there exists a critical probability $\vpc$ for the
existence of an infinite open self-avoiding walk that is
non-decreasing in each coordinate, and $\vpc(\Z^d) < 1$ for $d \ge
2$. The existence of certain types of open surface has also been
studied, see for example \cite{ACCFR,GG02,GH00,H00}.

The purpose of this note is to prove the existence of a
non-trivial critical point for the existence of a type of open
Lipschitz surface within site percolation on $\Z^d$ with $d \ge
2$. The existence of such surfaces is interesting in its own
right, and in addition there are several applications to be
developed elsewhere. We make a remark about the history of the
current note. Theorem \ref{main} was first proved by a subset
of the current authors, using an argument based on a
subcritical branching random walk, summarized in Section
\ref{sec:sbrw}. The simpler proof presented in Sections
\ref{sec:mainest} and \ref{sec:pf} was found subsequently by
the remaining authors.

Several applications and extensions of Theorem \ref{main} will
appear in \cite{D,DDS,GH09}. These include a study
of the movement of an interface through a field of obstacles,
and the existence of embeddings in $\Z^{d}$ of infinite words
indexed by $\Z^{d-1}$ (a problem posed by Ron Peled and
described in \cite{G-demon}).

\section{Local covers}\label{sec:mainthms}

We next state a variant of Theorem \ref{main} that is in a sense
stronger.  Let $d \ge 2$ and consider site percolation with
parameter $p$ on $\Z^d$.  Write $\Z_0^+=\{0,1,\ldots\}$. Let $x
\in \Z^{d-1}$. A \df{local cover} of $x$ is a function
$L:\Z^{d-1}\to \Z_0^+$ such that:
\begin{romlist}
\item for all $y\in\Z^{d-1}$, if $L(y)>0$ then
    $(y,L(y))$ is open;
\item for any $y,z\in \Z^{d-1}$ with $\|y-z\|=1$ we have
    $|L(y)-L(z)|\leq 1$;
\item $L(x)>0$.
\end{romlist}

If $x$ has a local cover, then the minimum of all local covers
of $x$ is itself a local cover of $x$; we call this the
\df{minimal} local cover of $x$ and denote it $L_x$. Define its
\df{radius}
$$
\rho_x:=\sup\Big\{\big\|(x,0)-(y,L_x(y))\big\|: y\in\Z^{d-1}
\text{ such that } L_x(y)>0\Big\},
$$
and take $\rho_x=\infty$ if $x$ has no local cover.

\begin{thm}\label{thm:hill}
For any $d\geq 2$ and $p\in(0,1)$ such that $q:=1-p < (2d)^{-2}$,
there exists
$A=A(p,d)<\oo$ such that
$$
\P_p(\rho_0\ge n) \le A[(2d)^2q]^n, \qq n\ge 0.
$$
\end{thm}

\section{Principal estimate}\label{sec:mainest}

The key step is to identify an appropriate set of dual paths that
are blocked by a Lipschitz surface of the type sought in Theorem
\ref{main}.  Such paths will be allowed to move downwards (that
is, in the direction of decreasing $d$-coordinate), with or
without a simultaneous horizontal move, but whenever they move
upwards, they must do so to a closed site.

Let $e_1,\ldots,e_d\in\Z^d$ be the standard basis vectors of
$\Z^d$. We define a \df{$\di$-path} from $u$ to $v$ to be any
finite sequence of distinct sites $u=x_0,x_1,\ldots,x_k=v$ of
$\Z^d$ such that for each $i=1,2,\ldots,k$:
\begin{equation}\label{lambdap}
x_i-x_{i-1}\in\{\pm e_d\}\cup \{- e_d\pm
    e_j:j=1,\ldots,d-1\}.
\end{equation}
A $\di$-path is called \df{admissible} if in addition for each
$i=1,2,\ldots,k$:
$$
\text{if }x_i-x_{i-1} = e_d\text{ then $x_i$ is closed.}
$$

Denote by $u\fr v$ the event that there exists an admissible
$\di$-path from $u$ to $v$, and write
$$
\tau_p(u)=\P_p(0\fr u).
$$
The next lemma is the basic estimate used in the proofs. For
$u= (u_1,u_2,\dots,u_d) \in \Z^d$, we write $h(u)=u_d$ for its
\df{height}, and
$$
r(u)= \|(u_1,u_2,\dots,u_{d-1})\| = \sum_{i=1}^{d-1} |u_i|.
$$
For $x\in\R$, $x^+ = \max\{0,x\}$ (\resp, $x^- = -\min\{0,x\}$)
denotes the positive (\resp, negative) part of $x$.

\begin{lemma}\label{sum2}
Let $d \ge 2$ and $a=2d$, and take $p\in(0,1)$ such that $q:=
1-p  \in(0,a^{-2})$. For $h\in\Z$ and $r\in \Z^+_0$ satisfying $r
\ge h^-$,
$$
\sum_{\substack{u\in\Z^d:\\ h(u)\ge h,\, r(u)\ge r}} \tau_p(u) \le
\frac 1{(1-aq)(1-a^2q)} (aq)^{h}(a^2q)^{r}.
$$
\end{lemma}

\begin{proof}
Fix $r\geq 0$, and let $h \in \Z$ satisfy $r \ge h^-$.
Let
$$
T=T_{r,h}=\{u\in \Z^d: h(u) \ge h,\ r(u)\ge r\}.
$$
Let $N(u)$ be the number of admissible $\di$-paths
(of all finite lengths) from $0$ to $u$, and note that
$$
\sum_{u\in T} \tau_p(u) = \sum_{u\in T} \P_p(N(u)>0) \leq
\sum_{u\in T} \E_p N(u).
$$

Let $\pi$ be a $\di$-path beginning at $0$.  Let $U$ and $D$ be the
respective numbers of steps in $\pi$ that lie in each of the
sets
$$
\{e_d\}; \quad \{-e_d\}\cup\{- e_d\pm
    e_j:j=1,\ldots,d-1\}.
$$
(The letters $U$, $D$ stand for `upwards' and `downwards'.)
Thus, the length of $\pi$ is $U+D$, final endpoint $u$ of $\pi$
satisfies $h(u)=U-D$ and $r(u) \le D$, and $\pi$ is admissible
with probability $q^U$, where $q:=1-p$. Also, the number of
$\di$-paths $\pi$ beginning at $0$ with given values of $U$ and
$D$ is at most $a^{U+D}$, where $a:=2d$.

Therefore,
$$
\sum_{u\in T} \E_p N(u)\leq
\sum_{\substack{U,D\ge 0: \\ U-D\geq h,\ D \ge r}}a^{U+D} q^{U}.
$$
Assume that $a^2q<1$ (i.e., $p>1-(2d)^{-2}$). Summing over
$U$, the last expression equals
\[
\frac{1}{1-aq} \sum_{D\ge r} a^D (aq)^{(h+D)^+}.
\]
Since $D \ge r \ge h^-$, we have $(h+D)^+ = h+D$, and the last sum
equals
\[
\sum_{D \ge r} (aq)^h (a^2q)^D = \frac{(aq)^h(a^2q)^r}{1-a^2q}.
\qedhere
\]
\end{proof}

\begin{remark}
The number of $\di$-paths of $k$ steps is no greater than
$(2d)^k$. Only minor changes are required to the proofs if one
restricts the class of $\di$-paths to those satisfying
\eqref{lambdap} for which $x_i - x_{i-1} \ne -e_d$ for all $i$.
The number of such paths is no greater than $(2d-1)^k$, and this
leads to improved versions of Theorems \ref{main} and
\ref{thm:hill} with $2d$ replaced by $2d-1$. The details are
omitted.
\end{remark}

\section{Proofs of Theorems \ref{main} and \ref{thm:hill}}\label{sec:pf}

We give two proofs of Theorem \ref{main}: one directly from Lemma
\ref{sum2}, and the other via Theorem \ref{thm:hill}. The second
proof gives a worse exponent in the inequality of Theorem
\ref{main}(iv).  We sketch a third approach in the next
section.

\begin{proof}[1st proof of Theorem \ref{main}]
Take $p$ and $a=2d$ as in Lemma \ref{sum2}.  Let
$T_-:=\Z^{d-1}\times\{\ldots,-1,0\}$ and define the random set of
sites
$$
G:=\{v\in\Z^d: u\fr v \text{ for some }u\in T_-\}.
$$
Since an admissible path may always be extended by a downwards
step (provided the new site is not already in the path), if $v\in
G$ then $v-e_d\in G$. Using Lemma \ref{sum2} with $r=0$, we have
for $h>0$ and suitable $A<\oo$,
\begin{equation}\label{height}
\P_p(h e_d\in G)\leq \sum_{u\in T_-} \P_p(u\fr h e_d) =\sum_{u \in
T_{0,h}}\tau_p(u)\leq A(aq)^h.
\end{equation}
Hence, by the Borel--Cantelli lemma, a.s.\ for every
$x\in\Z^{d-1}$, only finitely many of the sites $(x,h)=x+h e_d$
for $h>0$ lie in $G$.

For $x\in\Z^{d-1}$, let
$$
\lip(x):=\min \{t>0: (x,t)\notin G\}.
$$
The required properties (i) and (iii) of the theorem follow by
fact that $\P_p$ is a product measure, and (iv) is an immediate
consequence of \eqref{height}.  To check (ii), consider any
$x,y\in \Z^{d-1}$ with $\|x-y\|=1$.  Since $(x,\lip(x)-1)\in
G$, and an admissible path may be extended in the diagonal
direction $(y-x)-e_d$, we have $(y,\lip(x)-2)\in G$, whence
$\lip(y)> \lip(x)-2$.
\end{proof}

\begin{proof}[Proof of Theorem \ref{thm:hill}]
We begin with an explicit construction of the minimal local cover
$L_x$ of $x\in\Z^{d-1}$, whenever $x$ possesses a local cover. Let
$H_x$ be the set of endpoints of admissible paths from $(x,0)$
that use no site of $\Z^{d-1}\times \{-1,-2,\dots\}$. By
the definition of admissibility, $H_x$ does not depend on the
states of sites with height less than or equal to 0.

Let $a^2q<1$. By Lemma \ref{sum2}, $\rad(H_0)=\sup\{\|u\|: u\in
H_0\}$ satisfies
\begin{align*}
\P_p(\rad(H_0) \ge k) &\le \sum_{\substack{u\in \Z^d:\\ h(u)\ge 0,\
r(u)\ge k}} \tau_p(u)\\
&\le A(a^2q)^k,\qq k\ge 0,
\end{align*}
for some $A=A(p,d)<\oo$.

On the event that $|H_0|<\oo$, the minimal local cover of $0$
is given by $$L_x(y)=\min\{h\in\Z^+:(y,h)\notin H_0\};$$ (that
is, the corresponding surface consists of the sites immediately
above $H_0$).  The claim follows.
\end{proof}

\begin{proof}[2nd proof of Theorem \ref{main} with different exponent in part (iv)]
Let $p>1-(2d)^{-2}$, as in Theorem
\ref{thm:hill},  and let $H_x$ be as in the proof. Let
$$
F(x) := 1+\sup\{h: (x,h) \in H_y \text{ for some } y\in\Z^{d-1}\}.
$$
Given the general observations above, it suffices to prove that
$F$ satisfies part (iv) of Theorem \ref{main}. Now, for $k\ge 1$,
\begin{align*}
\P_p(F(0) > k) &=\P_p((0,k)\in H_y \text{ for some }y)\\
&\le \P_p(\rho_y\ge k+\|y\|\text{ for some }y)\\
&\le \sum_{y\in\Z^{d-1}} \P_p(\rho_0\ge k+\|y\|),
\end{align*}
and this decays to $0$ exponentially in $k$, by Theorem
\ref{thm:hill}.
\end{proof}

\section{Sketch proof using branching random walk}\label{sec:sbrw}

This section contains a summary of an alternative approach to
the problem, using a branching random walk to bound the size of
a minimal local cover. Write $\De=\Z^{d-1}\times\Z^+$, and
recall the \df{height} $h(x)$ of site $x$. The minimal cover
$L$ at the origin $0$ is in one--one correspondence with the
set $S:=\{(x,L(x)): L(x)>0\}$ of open sites.
The set $S$ may be constructed iteratively as
follows. Let $C$ be the height of the lowest open site above
$0$, that is, $C:= \inf\{n \ge 1: ne_d \text{ is open}\}$.
Clearly, $S$ contains no site of the form $(0,k)$, $1\le k <
C$, and in addition no site in the pyramid
$$
P:= \{x\in \Z^d : \|x\| < C\}.
$$
Let $x\in\De$ be such that $\|x\|=C$. If all such $x$ are open,
then $S=\{x \in \De: \|x\|=C)\}$.
Any such $x$ that is closed is regarded as a \df{child} of
the origin. Each such child $x$ is labelled with the height of
the lowest open site above it, that is, with the \df{label}
$h(x)+\inf\{n\ge 1: x+ne_d\text{ is open}\}$. The process is
iterated for each such child, and so on to later generations.
If the ensuing procedure terminates after a finite number of
steps, then we have constructed the set $S$ corresponding
to the minimal local cover of $0$.

A full analysis of the above procedure would require specifying
the order in which children are considered, as well as
understanding the interactions between different pyramids.
Rather than do this, we will treat the families of different
children as independent, thereby over-counting the total size
and extent of the process. That is, we construct a dominating
branching random walk, as follows.

Let $\xi=(\xi(z): z\in \Z)$ be a random measure on $\Z$ with
$\xi(z)\in \{0,1,2,\dots\}$ a.s. The corresponding branching
random walk begins with a single particle located at $0$, that is,
$\xi_0:=\delta_0$, the point mass. This particle produces
offspring $\xi_1:=\xi$. For $n \ge 2$, $\xi_n$ is obtained from
$\xi_{n-1}$ as follows: each particle of $\xi_{n-1}$ has
(independent) offspring with the same law as $\xi$, shifted
according to the position of the parent. Assume there exists
$\mu>0$ such that
\begin{equation}\label{alpha}
\alpha:=\E\biggl(
\sum_{z \in \Z} \ee^{\mu z} \xi(z) \biggr)<1,
\end{equation}
and define
$$
S_n:= \sum_{z \in \Z} \ee^{\mu z} \xi_n(z).
$$
It is standard that $S_n/\alpha^n$ is a (non-negative) martingale.
In particular, $S_n/\alpha^n$ converges a.s., whence $S_n \to 0$
a.s.\ as $n\to\oo$.

We next describe the law of $\xi$ arising in the current setting.
Let $Q$ be the set of all closed $x \in \Z^d$ satisfying $x\ne 0$
and
$$
\sum_{i=1}^{d-1}|x_i|=-x_d,
$$
and think of $Q$ as the set of children of the initial particle
at $0$. Each child is allocated a \df{location} in $\Z$ equal
to the height of the lowest open site above it. More precisely,
the location of the child corresponding to $x \in Q$ is defined
as $h(x)+\inf\{n\ge 1: x+ne_d \text{ is open}\}$, and
$\xi_n(z)$ is simply the number of children with location $z$.
The corresponding BRW is written \brwx.

The number of children with height $-n$ is binomially distributed
with parameters $(\tau_n,1-p)$, where $\tau_n \le 2(2n+1)^{d-1}$,
and the height of each tower has a geometric distribution.
Following an elementary calculation, there exist $\mu>0$ and
$p_1=p_1(d)\in(0,1)$ such that: for $p\in (p_1,1)$, we have
$\alpha<1$ in \eqref{alpha}.

We now compare \brwx\ and the local cover of $0$. With $C$ as
above, consider \brwx\ with all locations shifted by height $C$,
written \sbrwx. Each child of the origin in the percolation model
is a child in \sbrwx, and its label in the former equals its
location in the latter. The first generation of \sbrwx\ may also
contain children with negative heights. In subsequent generations,
the models are different, but it may be seen that \sbrwx\
dominates the percolation model in the sense that the set of
locations in \sbrwx\ with positive heights dominates
(stochastically) the set of labels in the percolation model.

With $\xi'_n$ the $n$th generation in \sbrwx, let
$$
N:= \sup\{n: \xi'_n(z)>0 \text{ for some } z>0\}.
$$
By the above domination, if $\P(N<\oo)=1$, then the local cover of
$0$ is (a.s.) finite. By Markov's inequality,
\begin{align*}
\P(\xi'_n(z)>0\text{ for some } z>0) &=\P(\xi_n(z)>0 \text{ for
some } z> -C)\\
& \le \E(S_n)\E(e^{\mu C})\\
&=\alpha^n \E(e^{\mu C}).
\end{align*}
By the Borel--Cantelli lemma, $\P(N<\oo)=1$, and the finiteness of
the local cover at $0$ follows.

Substantially more may be obtained by a more careful analysis of
the maximum displacement of particles in the $k$th generation of
\sbrwx. In particular, for $p>p_1$, one may deduce that
$\P_p(\rho_0 \ge n)$ decays to $0$ faster than a quantity that is
exponential in some power of $n$, and this implies the existence
of the Lipschitz function of Theorem \ref{main}, as in the
second proof of Section \ref{sec:pf}. The details of these arguments are
omitted.

\section*{Acknowledgements}
We thank Ron Peled for helping to bring the authors together, and
for indicating a minor error in a draft of this paper. Steffen Dereich
kindly suggested using the Laplace transform $\alpha$ in
the investigation of the branching random walk. G.\ Grimmett
acknowledges support from Microsoft Research during his stay as a
Visiting Researcher in the Theory Group in Redmond. P. Dondl and
M. Scheutzow acknowledge support from the DFG-funded research
group \lq Analysis and Stochastics in Complex Physical Systems'.

\bibliography{lip-perc} \bibliographystyle{amsplain}

\end{document}